\documentclass[intlimits,a4paper]{amsart}
\usepackage[utf8]{inputenc} 
\usepackage[all]{xy}
\usepackage{amssymb}

\newcommand{\Mat}{{\rm Mat}}

\newcommand{\sym}{\text{sym}}

\newcommand{\Q}{{\bf Q}}

\newcommand{\0}{\boldsymbol{0}}
\newcommand{\Z}{\mathbb{Z}}
\newcommand{\C}{\mathbb{C}}

\newcommand{\Hn}{\mathfrak{H}_n}
\newcommand{\fH}{\mathfrak{H}}
\newcommand{\fU}{\mathfrak{U}}

\newcommand{\cA}{{\mathcal A}}
\newcommand{\cJ}{{\mathcal J}}

\newcommand{\cR}{{\mathcal R}}

\newcommand{\cM}{{\mathcal M}}
\newcommand{\tM}{\widehat{{\rm Mat}}}

\newcommand{\R}{\mathbb{R}}

\newcommand{\rk}{{\rm rk}}

\newtheorem{theorem}{Theorem}[section]
\newtheorem{lemma}[theorem]{Lemma}

\newtheorem{remark}[theorem]{Remark}

\newtheorem{corollary}[theorem]{Corollary}

\newtheorem{proposition}[theorem]{Proposition}
\makeatletter
\newcommand{\raisemath}[1]{\mathpalette{\raisem@th{#1}}}
\newcommand{\raisem@th}[3]{\raisebox{#1}{$#2#3$}}
\makeatother

\begin{document}
\title[A decomposition of the Fourier-Jacobi coefficients]
{A decomposition of the Fourier-Jacobi coefficients of Klingen
  Eisenstein series}  
\author[T. Paul, R. Schulze-Pillot]{Thorsten Paul, Rainer Schulze-Pillot} 

\maketitle
\begin{abstract}
We investigate the relation between Klingen's decomposition of the
space of Siegel modular forms and Dulinski's analogous decomposition
of the
space of Jacobi forms.
\end{abstract}
\section{Introduction}
In analogy to the decomposition of the space of Siegel modular forms
of fixed weight and degree into the space of cusp forms and spaces of
Eisenstein series of Klingen type associated to cusp forms, Dulinski
showed in \cite{dulinski} that the space of Jacobi forms of fixed
weight, degree and index admits a natural decomposition into a direct
sum of the space of cusp forms and certain spaces of Jacobi Eisenstein
series of Klingen type. In \cite{boecherer_vasu},  Böcherer
studied how the Fourier-Jacobi coefficients of square free index of a
Klingen Eisenstein series of degree $2$ behave under this
decomposition, i.\ e., how one can identify the components in
Dulinski's decomposition of these Fourier-Jacobi coefficients. In
particular, whereas cusp forms have cuspidal Fourier-Jacobi
coefficients and the Siegel Eisenstein series has Siegel-Jacobi
Eisenstein series as Fourier-Jacobi coefficients, he showed  that the
Fourier-Jacobi coefficients of the Klingen Eisenstein series of degree
$2$ 
attached to  elliptic cusp forms have both a cuspidal and an
Eisenstein series part.

We continue this investigation here, using a different method, and
obtain an explicit description of the components for arbitrary degree
and index. Again, one sees that more than one component appears.

\medskip
This article and the talk at the RIMS workshop ``Automorphic Forms
and Related Topics'' in February 2017 on this topic by the
second author on which it is based give an overview of the work of the first author in his
doctoral dissertation \cite{pauldiss} written at Universität des
Saarlandes under the supervision of the second author. Most of the
proofs are only sketched, we refer to the dissertation for full
details. All results are
due to to the first author, the second author takes responsibility for
the present write-up and all possible mistakes in it.
We thank the RIMS and Prof. Nagaoka, who organized the workshop, for the
opportunity to present our work.

\section{Preliminaries}
For the basic notions of the theory of Siegel modular forms we refer
to \cite{freitagbook, klingenbook}, for Jacobi forms to 
\cite{dulinski}. In particular, we consider for
$k>n+1$ the
decomposition $\cM_n^k=\oplus_{m=0}^n\cM_{n,m}^k$ of the space of Siegel
modular forms of weight $k$ and degree $n$ for the full modular group
$Sp_n(\Z)$ into the spaces $\cM_{n,m}^k$ generated by Eisenstein series
$E_{n,m}^k(f)$ of Klingen type associated to a cusp form $f \in
\cM_m^k$.
For $F\in \cM_n^k$ we denote its Fourier coefficient at the symmetric
matrix $T$ by $A(F,T)$, here $T$ runs over the set
$\tM_n^\sym(\Z)$ of positive definite half integral symmetric
matrices of size $n$ with integral diagonal.

For $n'<n$ and $g=
\bigl(
\begin{smallmatrix}
  A&B\\C&D
\end{smallmatrix}\bigr)\in Sp_{n'}(\R)\subseteq GL_{2n'}(\R)$ we write
\begin{equation*}
  g^{\uparrow_n}=
  \begin{pmatrix}
    A&0&B&0\\
0&1_{n-n'}&0&0\\
C&0&D&0\\
0&0&0&1_{n-n'}
  \end{pmatrix},
g^{\downarrow_n}=
\begin{pmatrix}
 1_{n-n'}&0&0&0\\
0&A&0&B\\
0&0&1_{n-n'}&0\\
0&C&0&D 
\end{pmatrix},
\end{equation*}
for $U\in GL_n(\R)$ write $L(U)=\bigl(
\begin{smallmatrix}
  {}^tU^{-1}&0\\0&U
\end{smallmatrix}\bigr)\in Sp_{n}(\R)$.

We let $C_{n,r}\subseteq Sp_n(\Z)$ denote the intersection with
$Sp_n(\Z)$ of the
maximal parabolic $P_{n,r}(\Q)$ of $ Sp_n(\Q)\subseteq GL_{2n}(\Q)$
characterized as the set of  $g=(g_{ij}) \in Sp_n(\Q)$ with $g_{ij}=0$
  for $i>n+r, j\le n+r$ and 
$J_{n,r}\subseteq C_{n,r}$ (the Jacobi group of degree$(n,r)$) as the
set of elements of $C_{n,r}$ with an 
$(n-r)\times (n-r)$ identity matrix in the lower right hand corner. 
Notice that, with $n=r_1+r_2$,  Dulinski \cite{dulinski} writes $J^{r_1,r_2}\subseteq
C_{r_1+r_2,r_1}$ for this group. 

For $s \le r$ we divide an $n\times n$-matrix into blocks of sizes $\biggl(
\begin{smallmatrix}
  s\times s & s\times(r-s)&s\times (n-r)\\
(r-s)\times s&(r-s)\times (r-s)&(r-s)\times (n-r)\\
(n-r)\times s& (n-r)\times (r-s)&(n-r)\times (n-r)
\end{smallmatrix}\biggr)$ and let
\begin{equation*}
  Q_{s}^{r,n-r}=\{
  \begin{pmatrix}
    A&B\\C&D
  \end{pmatrix}\in Sp_n(\Z)\mid C=
            \begin{pmatrix}
              *&0&0\\
0&0&0\\
0&0&0
         \end{pmatrix}, D=
     \begin{pmatrix}
       *&*&*\\0&*&0\\0&*&1_{n-r}
     \end{pmatrix}\}.
\end{equation*}
With a block division of type
$\biggl(
\begin{smallmatrix}
  s\times s & s\times(n-r)&s\times (r-s)\\
(n-r)\times s&(n-r)\times (n-r)&(n-r)\times (r-s)\\
(r-s)\times s& (r-s)\times (n-r)&(r-s)\times (r-s)
\end{smallmatrix}\biggr)$
we let
\begin{equation*}
  \tilde{Q}_{s}^{r,n-r}=\{
  \begin{pmatrix}
    A&B\\C&D
  \end{pmatrix}\in Sp_n(\Z)\mid C=
            \begin{pmatrix}
              *&0&0\\
0&0&0\\
0&0&0
         \end{pmatrix}, D=
     \begin{pmatrix}
       *&*&*\\0&1_{r-s}&*\\0&0&*
     \end{pmatrix}\}.
\end{equation*}

For $n=r_1+r_2$ and $T \in \tM_{r_2}^\sym(\Z)$ we denote by
$\cJ^k_{r_1,r_2}(T)$ the space of Jacobi forms of weight $k$, degree
$(r_1,r_2)$ and index $T$ (which have good transformation behavior
under the Jacobi group $J_{n,r_1}$). A Siegel modular form then has a
Fourier-Jacobi expansion
\begin{equation*}
  F(Z)=\sum_{T_4 \in
    \tilde{M}_{r_2}^\sym(\Z)}\phi_{T_4}(z_1,z_2)e(T_4z_4)=\sum_{T_4}\phi^{(T_4)}(Z),
\end{equation*}
with Fourier-Jacobi coefficients  $\phi_{T_4}\in \cJ_{r_1,r_2}^k(T)$ of
degree $(r_1,r_2)$, index $T_4$  and weight $k$,
where $Z=
\bigl(\begin{smallmatrix}
 z_1&z_2\\{}^tz_2&z_4 
\end{smallmatrix}\bigr)$  is in the Siegel upper half plane $\Hn$ of degree $n$
with $z_1 \in \fH_{r_1},z_4\in \fH_{r_2}, z_2 \in \Mat_{r_1,r_2}(\C)$.

By Theorem 2 of \cite{dulinski} the space $\cJ^k_{r_1,r_2}(T)$ has a
decomposition
\begin{equation*}
  \cJ^k_{r_1,r_2}(T)=\bigoplus_{s=0}^{r_1}\cJ_{(r_1,r_2),s}^k(T), 
\end{equation*}
where the elements of $\cJ_{(r_1,r_2),s}^k(T)$ are Jacobi Eisenstein
series of Klingen type associated to Jacobi cusp forms of degree
$(s,r_2)$ with 
varying index $T'$ for which  $T'[U]=T$ for some integral matrix $U$.
Dulinski defines these Jacobi Eisenstein series of Klingen type only
for index $T$ of maximal rank. For $T$ of rank $t<r_2$ we notice that
by  \cite{ziegler} the space  $\cJ_{r_1,r_2}(T)$ is isomorphic to $\cJ_{r_1,r_2}(\bigl(
\begin{smallmatrix}
T_1&0\\0&0
\end{smallmatrix}\bigr))$ with a $T_1$ which is positive definite of
size $t$ and that this latter space is isomorphic to
$\cJ_{r_1,t}(T_1)$. These isomorphisms allow to transfer Dulinski's
definitions to index of arbitrary rank.

Our task is then to identify the components in this decomposition of
the Fourier-Jacobi coefficients of an Eisenstein series of Klingen
type as explicitly as possible. 
\section{Partial series of the Klingen Eisenstein series}
\begin{lemma}
 For $0\le m< n, 0\le r_1< n$ and $0\le t\le \min(n-m,n-r_1)$ let
 $M_{n,m,r_1}^t$ denote the set of all $g=
 \bigl(\begin{smallmatrix}
  A&B\\C&D 
 \end{smallmatrix}\bigr)
\in Sp_n(\Z)$ for which the lower right $(n-m)\times (n-r_1)$ block
$C_{22}$ of $C$ has rank $t$.

Then the sets $M_{n,m,r}^t$ are left $C_{n,m}$ and right
$C_{n,r}$-invariant, and for fixed $j,r$ their (disjoint) union over $0\le t
\le\min(n-m,n-r_1)$ is $Sp_n(\Z)$.
\end{lemma}
\begin{proof}
 This is easily checked, see the proof of Proposition 5.2 of \cite{pauldiss}.
\end{proof}
\begin{proposition} Let $f\in \cM_m^k$ be a cusp form.
  \begin{enumerate}
  \item For $0\le m<n, 0<r_1<n$ and $0\le t \le \min(n-m, n-r_1)$ the
    partial series 
    \begin{equation*}
      H_{n,m,r_1}^{t}(f;Z):=\sum_{\gamma \in
        C_{n,m}\backslash M_{n,m,r_1}^t}f(\gamma\langle Z\rangle^*)j(\gamma,Z)^{-k}
    \end{equation*}
of the Eisenstein series $E_{n,m}^k(f)$ of Klingen type is well defined
and invariant under the action $H \mapsto H|_kg$ of $g\in J_{n,r_1}$.
\item For
 $0\le m<n$ one has for each $r_1$ with $0\le r_1<n$ the decomposition
  \begin{equation*}
    E_{n,m}^k(f)=\sum_{t=0}^{\min(n-m,n-r_1)}H_{n,m,r_1}^{t}(f).
  \end{equation*}
\item The partial series  $H_{n,m,r_1}^{t}(f)$ has a Fourier-Jacobi
  decomposition
\begin{equation*}
      H_{n,m,r_1}^{t}(f;Z):=\sum_T\Psi_{n,m,r_1}^{(T),t}(f;Z)=\sum_T\Psi_{n,m,r_1;T}^{t}(f;z_1,z_2)e(Tz_4),
\end{equation*}
where the $\Psi_{n,m,r_1;T}^{t}(f;z_1,z_2)$ are Jacobi forms of degree
$(r_1,n-r_1)$ and index $T$.
  \end{enumerate}
\end{proposition}
\begin{proof}
  Obvious. The last assertion follows since both the existence of an
  expansion as given and the transformation behavior of the
  coefficients in it  hold for functions on $\fH_n$ which are
  $J_{n,r_1}$-invariant but not necessarily Siegel modular forms.
\end{proof}
\begin{remark}
Divide  a matrix $M\in \Mat_n(\R)$ for $0< m,r<n $ into blocks
$M_{11},M_{12}$, $M_{21},M_{22}$ of sizes  $j\times r, j\times
(n-r),(n-m)\times r, (n-m)\times(n-r)$ respectively.

For $\gamma=
\bigl(
\begin{smallmatrix}
  A&B\\C&D
\end{smallmatrix}\bigr)$ let $\gamma'$ be the $(n+m)\times (n+r)$
matrix obtained from $\gamma$ by removing the blocks
$A_{21},A_{22},B_{22},B_{12},D_{12}, D_{22}$ in the second block row and the
last block column. Then it can be shown (\cite[Satz 5.24]{pauldiss})
that the set $M_{n,m,r}^t$ is the set of all $\gamma \in Sp_n(\Z)$ for
which $\gamma'$ has rank $m+r+t$.
\end{remark}
In order to compute the partial series given above one needs explicit
coset representatives for $C_{n,m}\backslash
M_{n,m,r}^t$:
\begin{theorem}\label{coset_reps_1}
  Let $\cR_1^s$ for $s \le r$ denote a set of representatives of the double cosets
  in $L^{-1}(C_{m+r+t-2s,r-s})\backslash
  GL_{m+r+t-2s}^{r-s,*}(\Z)/L^{-1}(J_{m+r+t-2s,r-s})$ and $\cR_2^s$ a
  set of representatives of the cosets in 
  \begin{equation*}
  \begin{pmatrix}
    *&*\\0_{n-m-t+s-r,m+t+s}&*
  \end{pmatrix} \in GL_{n-r}(\Z)\}\backslash GL_{n-r}(\Z),
      \end{equation*}
  where $GL_{m+r+t-2s}^{r-s,*}(\Z)$ denotes the set of matrices in
  $GL_{m+r+t-2s}(\Z)$ for which the  $(r-s)\times (r-s)$ block in the
  lower left corner has full rank $r-s$.

For $u \in  GL_{m+r+t-2s}^{r-s,*}(\Z)$ we put 
\begin{equation*}
\hat{u}=
\begin{pmatrix}
  1_s&0&0\\
0&u&0\\
0&0&1_{n+s-m-t-r}
\end{pmatrix}
\in GL_n(\Z)
\end{equation*}
and for $u'\in GL_{n-r}(\Z)$ we
put $\tilde{u}'=\bigl(
\begin{smallmatrix}
1_r&0\\
  0&u'
\end{smallmatrix}\bigr) \in GL_n(\Z)$. 

Then a set of representatives of the cosets in $C_{n,m}\backslash
M_{n,m,r}^t$ is given by the matrices
\begin{equation*}
 \gamma_1^{\uparrow_n}L(\hat{u})\gamma_2L(\tilde{u}'), 
\end{equation*}
where for $s$ running from $\max(r+m+t-n,0)$ to $\min(j,r)$ one lets
$u$ run through $\cR_1^s$ and $u'$ through $\cR_2^s$, $\gamma_1$ runs
through a set of representatives for $C_{m+t,m}\backslash
M_{m+t,m,s}^t$ and $\gamma_2$ through a set of representatives of 
$$J_{m+t+r-s,r}^{\uparrow_n}\cap
L(\hat{u}^{-1})(\tilde{Q}_s^{r,m+t-s})^{\uparrow_n}L(\hat{u})\backslash J_{m+t+r-s,r}^{\uparrow_n}.$$
\end{theorem}
\begin{proof}
  This is Satz 5.21 of \cite{pauldiss}. The rather technical proof occupies most of
  Section 5.
\end{proof}
\section{The Fourier-Jacobi coefficients of the partial series}
\begin{lemma}
Let $f\in M_m^k$ be a cusp form. 
With the notations of Theorem \ref{coset_reps_1} let $s,u,u'$ be
 fixed and let $\gamma_1,\gamma_2$ run through the sets specified
 there.

Then the partial sum
\begin{equation*}
  \sum_{\gamma_1}\sum_{\gamma_2}
  f(\gamma_1^{\uparrow_n}L(\hat{u})\gamma_2L(\tilde{u}')\langle Z \rangle^*)j(\gamma_1^{\uparrow_n}L(\hat{u})\gamma_2L(\tilde{u}'),Z)^{-k}
\end{equation*}
has a Fourier-Jacobi expansion of degree $(r_1,r_2)$ with coefficients in
$\cJ_{(r_1,r_2),s}(T')$ whose index $T'$ has rank $m+t-s$.

In particular, for $m+t=n$ and $s=r_1$ the $T'$ occurring have maximal
rank $r_2$ and the Fourier-Jacobi coefficients are cusp forms.
\end{lemma}
\begin{proof}
 The first part of the assertion is formulated on p. 57 of \cite{pauldiss}
 before Lemma 6.3, its proof uses Lemma 6.3, 6.4, 6.6., where Lemma 6.6
 is the second part of our assertion.
\end{proof}
\begin{theorem}\label{fj_partialseries_decomp}
  \begin{enumerate}
  \item  The partial series  $ H_{n,m,r_1}^{t}(f)$ has a Fourier-Jacobi
 expansion whose  coefficient $\Psi(T):=\Psi_{n,m,r_1;T}^{t}(f)$ at $T$ is in
 $\cJ_{(r_1,r_2),m+t-\rk(T)}^k(T)$. 
\item Let $\phi(T)$  denote the Fourier-Jacobi coefficient at $T\in
  \widehat{\Mat}_{r_2}^\sym(\Z)$ of degree $(r_1,r_2)$ of the
    Eisenstein series $E_{n,m}(f)$ and let $\Psi(T)$ be as in i). 
Then $\Psi(T)$ is the component $\phi_{(r_1,r_2),m+t-\rk(T)}(T)$
of $\phi(T)$ in the space $\cJ_{(r_1,r_2),m+t-\rk(T)}^k(T)$ in
Dulinski's decomposition.    
  \end{enumerate}
\end{theorem}
  \begin{proof}
    The first assertion is proven in \cite{pauldiss} in the
    calculation following equation (6.2) on page 60 by using the lemma
    above and carrying out the summation over $u,u'$ from  the set of
    representatives given in Theorem \ref{coset_reps_1}.
The second assertion follows since the components in Dulinski's
decomposition are uniquely determined and $E_{n,m}(f)$ is the sum of
the partial series $H_{n,m,r_1}^t(f)$.
  \end{proof}
  \begin{remark}
 In particular, we see that only the spaces
 $\cJ_{(r_1,r_2),s}^k(T)$ with $m-\rk(T)  \le s \le
 \min(n-\rk(T),m+r_2-\rk(T))$.
For $m=n$ the lower bound and $\rk(T)\le r_2$ give $s\ge r_1$, hence
$s=r_1$, i.e., the Fourier-Jacobi coefficients of 
a cusp form are Jacobi cusp forms, which is trivial.

For $m=0$ we obtain $s \le r_2-\rk(T)$, so the Fourier-Jacobi
coefficients with index of maximal rank of
the Siegel
Eisenstein series are Jacobi 
Eisenstein series of Siegel type, which is known from
\cite{boecherer_fj}. For $\rk(T)<r_2$ the Fourier-Jacobi coefficient
of degree $(r_1,r_2)$ with index $T$ is essentially the Fourier-Jacobi 
coefficient of degree $(r_1, \rk(T))$ of the Siegel Eisenstein series of degree $n-(r_2-\rk(T))$
at a matrix of maximal rank, so it is again a  Jacobi 
Eisenstein series of Siegel type.
  \end{remark}
\section{Pullbacks and Fourier expansions}
Having identified the components in Dulinski's decomposition of the
Fourier-Jacobi expansion of the Eisenstein series $E_{n,m}(f)$ in
terms of the coefficients of the partial series $H_{n,m,r_1}^t$ we turn
now to the task of computing their Fourier expansion explicitly. For
this we adapt and refine ideas from \cite{boecherer_fj} to our
situation and divide the series defining the Siegel Eisenstein series
of degree $n+m$ into certain subseries in way similar to what we did
in Section 3.   

\begin{lemma}
Divide for $j,r \le n$ a matrix $M \in \Mat_{n+m}(\R)$ into blocks of types
\begin{equation*}
  \begin{pmatrix}
   j\times r&j\times(n-r)&j\times j\\
(n-m)\times r&(n-m)\times(n-r)&(n-m)\times j\\
j\times r&j\times(n-r)&j\times j 
  \end{pmatrix} 
\end{equation*}
and denote these blocks by $M_{11},\dots,M_{33}$.
For $\gamma = \bigl(
\begin{smallmatrix}
  A&B\\C&D
\end{smallmatrix}\bigr)
\in Sp_{n+m}(\Z)$ let $\hat{\gamma}=\biggl(
\begin{smallmatrix}
  C_{11}&C_{12}&D_{11}\\
C_{21}&C_{22}&D_{21}\\
C_{31}&C_{32}&D_{31}
\end{smallmatrix}\biggr) \in \Mat_{n+m,n+r}(\Z)$ and denote for $m+r
\le v\le
\min(n+m,n+r)$ the set of all $\gamma \in Sp_{n+m}(\Z)$ with
$\rk(\hat{\gamma})=v$ by $X_{n,m,r}^v$.

then  $X_{n,m,r}^v$ is left invariant under $C_{n+m,0}$ and right
invariant under $Sp_m^{\downarrow_{n+m}}(\Z)$, and $Sp_{n+m}(\Z)$ is
  the disjoint union  of the  $X_{n,m,r}^v$ for $m+r\le v \le \min(n+m,n+r)$.  
\end{lemma}
\begin{proof}
  This is Proposition 7.2 of \cite{pauldiss}. Since $\hat{\gamma}$ is
  obtained from $\gamma$ by deleting $n-r$ columns and $n-m$ rows, its
  rank $v$ must be between $m+r$ and $\min(n+m,n+r)$. the assertions
  about left and right invariance are  checked easily.
\end{proof}
We need an explicit set of representatives of the cosets in
$C_{n+m,0}\backslash X_{n,m,r}^v$. For this we recall that by
\cite{garrett} a set of representatives for $C_{n+m,0}  \backslash
Sp_{n+m}(\Z)$ is given by the products
\begin{equation*}
  g_{j,M}(g_{j,0}')^{\uparrow_{n+m}}g_j'^{\uparrow_{n+m}}((g_{j,1}'')^{\uparrow_m})^\downarrow_{n+m}(g_j'')^\downarrow_{n+m},
\end{equation*}
where $j$ runs from $0$ to $m$, and for any such $j$ we let $g'_{j,0}$
run through $Sp_j(\Z)$, $g'_j$ through a set of representatives for
$C_{n,j}\backslash Sp_n(\Z)$ and $g''$ through a set of
representatives for $C_{m,j}\backslash Sp_m(\Z)$. Moreover, with $M'$
running through the $j\times j$ elementary divisor matrices and $M=\bigl(
\begin{smallmatrix}
  M'&0\\
0&0
\end{smallmatrix}\bigr) \in M_{m,n}(\Z)$ we let $g_{j,M}=\biggl(
  \begin{smallmatrix}
    1_n&0&0&0\\
0&1m&0&0\\
0&{}^tM&1_n&0\\
M&0&0&1_m
  \end{smallmatrix}\biggr)$ and $\Gamma_j(M'):=Sp_j(\Z) \cap \bigl(
  \begin{smallmatrix}
    0&M'^{-1}\\M'&0
  \end{smallmatrix}\bigr) Sp_s(\Z) \bigl(
  \begin{smallmatrix}
    0&M'^{-1}\\M'&0
  \end{smallmatrix}\bigr)$ and let $g_{j,1}''$ run through a set of
  representatives of $\Gamma_j(M')\backslash Sp_j(\Z)$.

\begin{proposition}
 A set of representatives for $C_{n+m,0}\backslash X_{n,m,r}^v$ is
 obtained from the representatives above by restricting $g_j' to$ a
 set of representatives of $C_{n,j}\backslash   M_{n,j,r}^{v-r-j}$.
\end{proposition}
\begin{proof}
  A straightforward computation shows that indeed these are precisely
  the products which are in $X_{m,n,r}^v$, see Satz 7.4 of
  \cite{pauldiss} and the proof given there.
\end{proof}
\begin{theorem}\label{pullback_gnmrv}
For $0< s \le m$ let $(f_{s,\nu})_\nu$ be an orthonormal basis of
Hecke 
eigenforms for the space of cusp forms of degree $s$ and weight $k$. 
 We set 
 $A_s^k:=\pi^{\frac{s(s-1)}{4}}(4\pi)^{\frac{s(s+1)}{2}-sk}\prod_{i=1}^s\Gamma(k-\frac{s+i}{2})$
 and 
  \begin{equation*}
   \beta(s,k)=(-1)^{\frac{sk}{2}}s^{s(k-\frac{s-1}{2})}\prod_{i=0}^{s-1}\frac{\pi^{k-\frac{i}{2}}}{\Gamma(k-\frac{i}{2})}\zeta(k)^{-1}\prod_{i=1}^m\zeta(2k-2i)^{-1}.
 \end{equation*}
  For $0\le m,r <n$ and $m+r \le v \le \min(n+m,n+r)$ we put 
  \begin{equation*}
    G_{n,m,r}^v(Z):=\sum_{\gamma \in C_{n+m,0}\backslash X_{n,m,r}^v}
      j(\gamma, Z)^{-k}.
  \end{equation*}
Then for $Z_1 \in \fH_{n},Z_2 \in \fH_m$ the pullback $G_{n,m,r}^v(\bigl(
\begin{smallmatrix}
  -\overline{Z_1}&0\\0&Z_2
\end{smallmatrix}\bigr))$ of $G_{n,m,r}^v$ to $\fH_{n}\times \fH_m$
can be written as 
\begin{equation*}
 G_{n,m,r}^v
\left(\begin{pmatrix}
  -\overline{Z_1}&0\\0&Z_2
\end{pmatrix}\right)= \sum_{s=0}^m
c_s\sum_{\nu}D_{f_{s,\nu}}(k-s) E_{m,s}(f_{s,\nu};Z_2)\overline{H_{n,s,r}^{v-r-s}(f,Z_1)}, 
\end{equation*}

where $D_{f_{s,\nu}} $ denotes the standard $L$-function of the Hecke
eigenform $f_{s,\nu}$ (and this factor doesn't occur for $s=0$) and
where for $s>0$ we put $c_s=2 \beta(s,k)A_s^k$ and set $c_0=1$.  
\end{theorem}
\begin{proof}
This follows from the proof of the theorem in Section 5 of
\cite{garrett}   and the explicit evaluation of the constants
occurring there in \cite{boecherer_fj}.
\end{proof}
 \begin{corollary}
 For a Hecke eigenform $f \in M_m^k$ of Petersson norm $1$ one has 
 \begin{equation*}
   H_{n,m,r}^{v-r-m}(f;Z_1)=\lambda(f)^{-1}\left\langle f(\cdot),G_{n,m,r}^v(
   \begin{pmatrix}
     -\overline{Z_1}&0\\0&\cdot
   \end{pmatrix})\right\rangle
 \end{equation*}
with $\lambda(f)=2 \beta(m,k)A_m^kD_f(k-m) $ as in the theorem above.
\end{corollary}
\begin{proof}
  This follows since taking the Petersson product with $f$ singles out the summand
  containing $H_{n,m,r}^{v-r-m}(f,Z_1)$ from the formula in the theorem.
\end{proof}
By the corollary we can compute the Fourier expansion of our partial
series $H_{n,m,r}^{v-r-m}(f)$ by computing the Petersson product on
the right hand side. We will do this adapting   again ideas from
\cite{boecherer_fj}.
\begin{lemma}\label{coset_reps_2}
  \begin{enumerate}
  \item  
Let $P_{n,m}=\bigl(
\begin{smallmatrix}
  0&1_m\\1_n&0
\end{smallmatrix}\bigr)$. Then for $l\le n$ the set
$M_{n+m,0,n}^lL(P_{n,m})\cap X_{n,m,r}^v$ is nonempty only if $l\le v$
and $X_{n,m,r}^v$ is contained in the (disjoint) union of the
$M_{n+m,0,n}^lL(P_{n,m})$ for $0\le l \le v$.
\item With 
$$G_{n,m,r}^{v,l}(Z):=\sum _{M_{n+m,0,n}^lL(P_{n,m})\cap
  X_{n,m,r}^v}j(\gamma,Z)^{-k}$$
one has $G_{n,m,r}^v(Z)=\sum_{l=0}^vG_{n,m,r}^{v,l}(Z)$ .

\item A set of representatives of $C_{n+m,0}\backslash
M_{n+m,0,n}^lL(P_{n,m})\cap X_{n,m,r}^v$is given by the\\
$x^{\uparrow_{n+m}}L(U)y^{\downarrow_{n+m}}$, where $x$ runs through a set of
representatives of $C_{l.0}\backslash M_{l,0,0}^l$, $y$ through a set
of representatives of $C_{m,0}\backslash Sp_m(\Z)$ and $U$ through a
set of representatives of 
\begin{equation*}
  \left\{
  \begin{pmatrix}
    *&*&*\\
0_{n-l,l}&*&*\\
0_{m,l}&*&*
  \end{pmatrix}\right\} \backslash \left\{
  \begin{pmatrix}
    u_1&u_2&u_3\\
u_4&u_5&u_6\\
u_7&u_8&u_9
  \end{pmatrix}\in GL_{n+m}(\Z)\mid \rk
  \begin{pmatrix}
    u_4\\u_7
  \end{pmatrix}=v-l,\rk
  \begin{pmatrix}
    u_6\\u_9
  \end{pmatrix}=m\right\},
\end{equation*}
where  $U$ has a block division of type $\biggl(
\begin{smallmatrix}
  l\times r&l\times(n-r)&l\times m\\
(n-l)\times r&(n-l)\times(n-r)&(n-l)\times m\\
m\times r&m\times(n-r)&m\times m
\end{smallmatrix}\biggr)$.
 \end{enumerate}
\end{lemma}
\begin{proof}
  This is Satz 8.1 of \cite{pauldiss}.  For the proof one checks which
  of the representatives of $C_{n+m,0}\backslash M_{n+m,0,n}^l
  L(P_{n,m})$ obtained from Theorem \ref{coset_reps_1} are in $X_{n,m,r
  }^v$, see \cite{pauldiss} for details.
\end{proof}
\begin{lemma}
 Let $U$ run through the  set of representatives  from the previous
 lemma and write a matrix in $\Mat_{n+m,l}(\Z)$ as $\biggl(
 \begin{smallmatrix}
   w_1\\w_2\\w_3
 \end{smallmatrix}\biggr)$, where $w_1,w_2,w_3$ have $r,n-r,m$ rows respectively.
Then  the matrix formed by the first $l$ columns of
 $U^{-1}$ runs through a set of representatives of 
 \begin{equation*}
 \left\{   \begin{pmatrix}
   w_1\\w_2\\w_3
 \end{pmatrix} \text{ primitive }\mid \rk
 \begin{pmatrix}
   w_1\\w_2
 \end{pmatrix}=l, \rk
 \begin{pmatrix}
   w_2\\w_3
 \end{pmatrix}=v-r\right\}/GL_l(\Z).
 \end{equation*}
\end{lemma}
\begin{proof}
 This is Lemma 8.4 a) of \cite{pauldiss}. The proof uses computations
 from Lemma 5.7 and Remark 5.8 of \cite{pauldiss}.
\end{proof}
\begin{lemma}\label{gnmrvl_firstformula}
We denote by  $a_l(T)$ the Fourier coefficient at $T$ of the Siegel
Eisenstein series of degree $l$ and weight $k$ and write $\cA_l^+$ for
the set of positive definite matrices in
$\widehat{\Mat}_l^{\sym}(\Z)$. Then
  \begin{equation*}
    G_{n,m,r}^{v,l}(Z)=\sum_{T\in
      \cA_l^+}\sum_{w_1,w_2}\sum_{w_3}\sum_{y}a_l(T)e(T\left((y^{\downarrow_{n+m}}\langle Z\rangle)^*\left[
    \begin{pmatrix}
      w_1\\w_2\\w_3
    \end{pmatrix}\right]\right))j(y,z_4)^{-k},
  \end{equation*}
where $y$ runs through a set of representatives of 
$C_{m,0}\backslash
      Sp_m(\Z)$, $w_1,w_2,w_3$ are as in the previous lemma, and  $z_4 \in \fH_m$
is the lower left $m\times m$ corner of $Z$. 
\end{lemma}
\begin{proof}
 We carry out the summation over the coset representatives given in
 Lemma \ref{coset_reps_2}, expanding the automorphy factor $j$ using
 its cocycle relation and $j(L(U),\cdot)=1$. The summation over $x$
 gives then by \cite[Lemma 3]{boecherer_fj}
 \begin{equation*}
  \sum_{T\in \cA_l^+}\sum_U\sum_y
    a_l(T)e(T(L(U)y^{\downarrow_{n+m}}\langle Z\rangle)^*)j(Y,z_4)^{-k}, 
 \end{equation*}

Using $L(U)y^{\downarrow_{n+m}}\langle
Z\rangle=y^{\downarrow_{n+m}}\langle Z \rangle [U^{-1}]$ and writing
the upper left block of $U^{-1}$ in terms of $w_1,w_2,w_3$ as in the
previous lemma, we obtain the assertion.
\end{proof}
\begin{lemma}
Write $\Z_s^{m\times l}=\{w\in \Mat_{m,l}(\Z)\mid \rk(w)=s\}$,
$\Z_{s,0}^{m\times l} =\{ \bigl(
\begin{smallmatrix}
  *\\0_{m-s,l}
\end{smallmatrix}\bigr)\in \Z_s^{m\times l}\}$. \\
Let $GL_m(\Z)_s=\{\bigl(
\begin{smallmatrix}
  *&*\\0_{m-s,s}&*
\end{smallmatrix}\bigr)\in GL_m(\Z)\}$ and 
 $GL_m(\Z)_s^1=\{\bigl(
\begin{smallmatrix}
 1_s&*\\0_{m-s,s}&*
\end{smallmatrix}\bigr)\in GL_m(\Z)\}$.\\
Let $w_3'$ run through a set of representatives of
$GL_m(\Z)_s^1\backslash \Z_{s,0}^{m\times l}$ and $w_3''$ through a
set of representatives of $GL_m(\Z)/GL_m(\Z)_s^1$. Then every element
of $Z_s^{m\times l}$ has a unique expression as a product $w_3''w_3'$,
and all these products are in $ Z_s^{m\times l}$.

For $w_1,w_2$ fixed, the matrix $\biggl(
\begin{smallmatrix}
  w_1\\w_2\\w_3'
\end{smallmatrix}\biggr)$
is primitive if and only if $\biggl(
\begin{smallmatrix}
  w_1\\w_2\\w_3''w_3'
\end{smallmatrix}\biggr)$ is primitive, and one has $\rk\bigl(
\begin{smallmatrix}
  w_2\\w_3'
\end{smallmatrix}\bigr)=\rk\bigl(
\begin{smallmatrix}
  w_2\\w_3''w_3'
\end{smallmatrix}\bigr)$. 
\end{lemma}
\begin{proof}
This is Lemma 8.4 b) of \cite{pauldiss}. It is clear that any $u\in
\Z_s^{m\times l}$ can be written as $w w_3'$ with $w \in GL_m(\Z)$ and
$w_3'\in Z_{s,0}^{m\times l}$, where $w_3'$ is unique up to
multiplication with an element of $GL_m(\Z)_s$ from the
left. Moreover, if $w_3'$ is fixed, $w$ is unique up to right
multiplication by an element of $GL_m(\Z)_s^1$.
The second assertion is obvious.
\end{proof}
\begin{lemma}
  \begin{enumerate}
  \item 
With notations as in Lemma \ref{gnmrvl_firstformula} the sum
\begin{equation*}
  \sum_{w_3}e(T\left((y^{\downarrow_{n+m}}\langle Z\rangle)^*\left[
    \begin{pmatrix}
      w_1\\w_2\\w_3
    \end{pmatrix}\right]\right))j(y,z_4)^{-k}
\end{equation*} for $T,y,w_1,w_2$ fixed is equal to
\begin{equation*}
  \sum_s\sum_{w_3'}\sum_{w_3''}e(T\left((L(w_3''^{-1})^{\downarrow_{n+m}}y^{\downarrow_{n+m}}\langle Z\rangle)^*\left[
    \begin{pmatrix}
      w_1\\w_2\\w_3'
    \end{pmatrix}\right]\right))j(y,z_4)^{-k},
\end{equation*}
where $s$ runs from $0$ to $\min(l,m)$, $w_3'$ runs over the set of
matrices in $\Z_s^{m\times l}$ for which $\bigl(
\begin{smallmatrix}
  w_2\\w_3'
\end{smallmatrix}\bigr)$ has rank $v-r$, and $w_3''$ runs over a set
of representatives of $GL_m(\Z)/GL_m(\Z)_s^1$. 
\item For a block diagonal matrix $Z=\bigl(
  \begin{smallmatrix}
    Z_1&0\\0&Z_2
  \end{smallmatrix}\bigr)$ with $Z_1\in \fH_n,Z_2\in \fH_m$ one has 
  \begin{equation*}
    G_{n,m,r}^{v,l}(Z)=\sum_{T\in \cA_l^+}a(T)\sum_{w_1,w_2}e(T\left[{}^t
      \begin{pmatrix}
        w_1\\w_2
      \end{pmatrix}\right]Z_1)\sum_{s=0}^{\min(l,m)}\epsilon(s)\sum_{w_3'}g_{m,s}^k(Z_2,T[{}^tw_3']),
  \end{equation*}
with $\epsilon(0)=1$ and $\epsilon(s)=2$ otherwise, where the
summations over $w_1,w_2,w_3'$ are as before and where the 
Poincar\'e series $g_{m,s}^k(Z_2,T[{}^tw_3'])$ is given by
\begin{equation*}
  g_{m,s}^k( Z_2,T_1^\uparrow)=\sum_{\gamma \in \fU_{m,s}\backslash
    Sp_m(\Z)}e(T_1^\uparrow (\gamma\langle Z\rangle))j(\gamma,Z)^{-k},
\end{equation*} 
where
$\fU_{m,s}\subseteq C_{m,0}$ is the group of matrices $\left(
  \begin{smallmatrix}
    A&B\\0&D
  \end{smallmatrix}\right) \in C_{m,0}\subseteq Sp_m(\Z)$, with $A=
\left(  \begin{smallmatrix}
    \pm 1_s&*\\0&*
  \end{smallmatrix}\right)$.
  \end{enumerate}
\end{lemma}
\begin{proof}
For a) we use the decomposition $w_3=w_3''w_3'$ from the previous
lemma and order the sum over $w_3'$ by the rank $s$ of $w_3'$.
For b),
with $\fU_{m,s}^+=\{\left(
  \begin{smallmatrix}
    A&B\\0&D
  \end{smallmatrix}\right) \in \fU_{m,s}\mid A=
\left(  \begin{smallmatrix}
     1_s&*\\0&*
  \end{smallmatrix}\right)\}$ we see that $L(w_3''^{-1})$ runs through a set
of representatives of $\fU_{m,s}^+\backslash C_{m,0}$, so that
$L(w_3''^{-1})y$ runs through a set of representatives $\tilde{y}$ of
$\fU_{m,s}^+\backslash Sp_m(\Z)$ which satisfy
$j(\tilde{y},Z_2)=j(y,Z_2)$ for $\tilde{y}=L(w_3''^{-1})y$ and $Z_2\in
\fH_m$. For $s=0$ one has
$\fU_{m,s}=\fU_{m,s}^+$, for $s>0$ each coset modulo $\fU_{m,s}$ is
  the union of two cosets modulo $\fU_{m,s}^+$, which explains the
  factor $\epsilon(s)$. The expression obtained in a) then transforms
  (with $z_4=Z_2$) to
  \begin{equation*}
    a_l(T)e(T[{}^{\raisemath{3pt}{t}}\!\!
    \begin{pmatrix}
      w_1\\w_2
    \end{pmatrix}
]Z_1)\sum_s\epsilon(s)\sum_{w_3'}
\sum_{\tilde{y}\in\fU_{m,s}\backslash Sp_m(\Z)}
e(T[{}^tw_3']\tilde{y}\langle Z_2\rangle)j(\tilde{y},Z_2)^{-k},
  \end{equation*}
and the  sum over $\tilde{y}$ equals the Poincar\'e series
$g_{m,s}^k(Z_2,T[{}^tw_3'])$ (notice that $T[{}^tw_3']$ has the block
diagonal shape required).
\end{proof}
\begin{theorem} Let $f(Z)=\sum_{S\in \widehat{\Mat}_m^\sym(\Z)}b(S)e(SZ)\in M_m^k$ be a cusp form with
    Fourier coefficients $b(S)$.

Then the  Fourier coefficient of $H_{n,m,r}^t(f)$ at $R\in
\widehat{\Mat}_n^\sym(\Z)$ with $\rk(R)=l$ is
\begin{equation*}
  \beta(m,k)^{-1}D_f(k-m)^{-1}\sum_{T\in \cA_l^+}\sum_{w_1,w_2}\sum_{w_3'}b(T[{}^tw_3'])\det(T[{}^tw_3'])^{\frac{m+1}{2}-k}
\end{equation*}
with $\beta(m,k), D_f(k-m)$ as in Theorem \ref{pullback_gnmrv}.

In the sum, $\biggl(
\begin{smallmatrix}
  w_1\\w_2\\w_3'
\end{smallmatrix}\biggr)\in \Mat_{n+m,l}(\Z)$ with $w_1 \in
\Mat_{r,l}(\Z),w_2\in \Mat_{n-r,l}(\Z), w_3'\in \Mat_{m,l}(\Z)$ runs
through those primitive elements of a set of representatives of $
\Mat_{n+m,l}(\Z)/GL_l(\Z)$ which satisfy
\begin{equation*}
  R=T[{}^{\raisemath{3pt}{t}}\!\!
    \begin{pmatrix}
      w_1\\w_2
    \end{pmatrix}
],\quad \rk(w_3')=m,\quad \rk
\begin{pmatrix}
  w_2\\w_3'
\end{pmatrix}=t+m.
\end{equation*}
\end{theorem}
\begin{proof}
 By our previous results only the Petersson product $\langle f(\cdot),G_{n,m,r}^{t+m+r,l}\bigl(
 \begin{smallmatrix}
   -\overline{Z_1}&0\\0&\cdot
 \end{smallmatrix}\bigr)\rangle$
contributes to the Fourier coefficient of $H_{n,m,r}^t(f)$ at a matrix
$R$ of rank $l$, and we have reduced the computation of this Petersson
product to the product with the Poincar\'e series $g_{m,s}^k(Z_2,T[{}^tw_3'])$. For
$s<m$, these are known to be orthogonal to cusp forms (being
Eisenstein series of Klingen type), for $s=m$ the Petersson product
has been computed in \cite[p.90,94]{klingenbook}. Plugging in that
result gives the assertion.
\end{proof}
\begin{remark}
  It should be noticed that the sum in the formula of the theorem is a
  finite sum.
\end{remark}
\begin{corollary}
  As in Theorem \ref{fj_partialseries_decomp} denote by $\phi_{m+t-\rk(R_4)}^{(R_4)}$ the
 component in $\cJ_{(r_1,r_2),m+t-\rk(T)}^k(T)$ of the   Fourier-Jacobi coefficient at the $r_2\times r_2$-symmetric matrix $R_4$
  of the Klingen Eisenstein series $E_{n,m}^k(f)$.

Then the Fourier coefficient at $(R_1,R_2)$ of
$\phi_{m+t-\rk(R_4)}^{(R_4)}$ is given by the formula in the previous
theorem for the Fourier coefficient of $H_{n,m,r_1}^t(f)$ at $R=\bigl(
\begin{smallmatrix}
  R_1&R_2\\{}^tR_2&R_4
\end{smallmatrix}\bigr)$.
\end{corollary}
\begin{proof}
This follows directly from the previous theorem and Theorem \ref{fj_partialseries_decomp}.  
\end{proof}
\section{The case $n=2$}
We consider here $r=r_1=r_2=m=1$, i.e., we study the Klingen
Eisenstein series attached to an elliptic cusp form
$f(z)=\sum_{n=1}^\infty b(n)e(nz)$, which we
assume to be a Hecke eigenform.

One obtains here
$\beta(m,k)^{-1}D_f(k-1)^{-1}=\frac{1}{2}\zeta(1-k)\zeta(2k-2)L_2(f,2k-2)^{-1}$,
where $ L_2(f,s)=\zeta(2s-sk+2)\sum_{n=1}^\infty b(n^2)n^{-s}$ is the
symmetric square $L$-function of $f$. We have to consider the
$H_{2,1,1}^t$ for $t=0,t=1$. For $t=1$ our computation in the previous
paragraph shows that $H_{2,1,1}^1$ has  nonzero Fourier coefficients
only at matrices $R=\bigl(
\begin{smallmatrix}
  r_1&\frac{r_2}{2}\\
\frac{r_2}{2}&r_4
\end{smallmatrix}\bigr)$
 of rank $2$. The Fourier coefficient at such an $R$ is then computed
 as
 \begin{equation*}\begin{split}
  \frac
  {1}{2}\zeta(1-k)&\zeta(2k-2)L_2(f,2k-2)^{-1}\sum_{a,b,d}a_2(T)\\
&\times \sum_{u,v}b(u^2t_1+uvt_2+v^2t_4)(u^2t_1+uvt_2+v^2t_4)^{1-k},
 \end{split}\end{equation*}
where the summation over $a,b,d$
runs over $a,d>0$ and $0\le b<a$ such that 
\begin{equation*}
  T=
  \begin{pmatrix}
    t_1&\frac{t_2}{2}\\\frac{t_2}{2}&t_4
  \end{pmatrix}=
  \begin{pmatrix}
    a&b\\0&d
  \end{pmatrix}R\begin{pmatrix}
    a&b\\0&d
  \end{pmatrix}^{-1}\in \widehat{\Mat}_2^\sym(\Z)
\end{equation*} and the summation over $u,v$ runs over $u,v \in \Z$
satisfying $u \ne 0, \gcd(u,a)=\gcd(av-ub,d)=1$. If $-\det(2R)$ is a
fundamental discriminant only $a=d=1$ occurs,  and  one checks that this agrees with the result
in \cite{boecherer_vasu}. One can proceed from here to obtain
asymptotic formulas as in \cite{boecherer_vasu}. For details see
\cite[Section 9]{pauldiss}.

Authors:

\smallskip
Thorsten Paul\\
Rainer Schulze-Pillot, FR Mathematik, Universität des Saarlandes,
Postfach 151150, 66041 Saarbrücken, Email: schulzep@math.uni-sb.de

\end{document}